\documentclass[10pt,a4paper]{amsart}

\usepackage[all]{xy}
\usepackage[normalem]{ulem}
\usepackage{amsmath}
\usepackage{amsfonts}
\usepackage{amssymb}
\usepackage{amsthm}
\usepackage{graphicx}
\usepackage{microtype}
\usepackage[usenames,dvipsnames]{xcolor}
\usepackage{pgf,tikz}
\usepackage{mathrsfs}
\usetikzlibrary{arrows}
\usepackage{caption}
\usepackage{soul}
\usepackage[pdftex,plainpages=false,bookmarks=true,breaklinks=true,linkcolor=brown,citecolor=brown,colorlinks]{hyperref}
\usepackage{bm}
\usepackage{cancel} 
\usepackage{combelow}

\newcommand{\prarrow}[2]{\ar@<0.5ex>[r]^-{#1} \ar@<-0.5ex>[r]_-{#2}}
\newcommand{\plarrow}[2]{\ar@<0.5ex>[l]^-{#1} \ar@<-0.5ex>[l]_-{#2}}
\newcommand{\pdarrow}[2]{\ar@<0.5ex>[d]^-{#1} \ar@<-0.5ex>[d]_-{#2}}
\newcommand{\puarrow}[2]{\ar@<0.5ex>[u]^-{#1} \ar@<-0.5ex>[u]_-{#2}}

\newtheorem{theorem}{Theorem}[section]
\newtheorem{corollary}[theorem]{Corollary}    
\newtheorem{lemma}[theorem]{Lemma}    
\newtheorem{proposition}[theorem]{Proposition}
\newtheorem{questions}[theorem]{Questions}
\newtheorem{question}[theorem]{Question}
\theoremstyle{definition}
\newtheorem{definition}[theorem]{Definition} 
\newtheorem{example}[theorem]{Example}

\author{Eiichi Matsuhashi}
\address{Department of Mathematics, Shimane University, Matsue, Shimane 690-8504, Japan.}
\email{matsuhashi@riko.shimane-u.ac.jp}
\title{Some decomposable continua and Whitney levels of their hyperspaces}

\author{Yoshiyuki Oshima}
\address{Department of Mathematics, Shimane University, Matsue, Shimane 690-8504, Japan.}
\email{n20d101@matsu.shimane-u.ac.jp}
\begin{document}

\begin{abstract}
 
We introduce the new class of continua; {\it $D^{**}$-continua}. The classes of   Wilder continua and   $D^{*}$-continua are strictly contained in the class of  $D^{**}$-continua. Also, the class of  $D$-continua is bigger than the class of $D^{**}$-continua. Using  $D^{**}$-continua,   we give the negative answer to \cite[Question 2]{D}. Furthermore, we prove that being Wilder, being $D$, being $D^*$ and being $D^{**}$ are Whitney properties. 

\end{abstract}

\keywords{Wilder continuum, $D$-continuum, $D^*$-continuum, $D^{**}$-continuum,  decomposable continuum, Whitney map, Whitney property}
\subjclass[2020]{Primary  54F16 ; Secondary 54F15}      

\maketitle

\markboth{Some decomposable continua and Whitney levels}{Eiichi Matsuhashi and Yoshiyuki Oshima}

\section{Introduction}
In this paper, all spaces are  metrizable spaces and  maps are continuous. If $X$ is a space and $A$ is a subset of $X$, then we denote the interior of $A$ in  $X$ by ${\rm Int}_X A$. 
 By a $continuum$ we mean a compact connected metric space.

Let $X$ be
a continuum.  Then, $C(X)$ denotes the  space of all nonempty subcontinua of $X$ endowed with the Hausdorff metric. $C(X)$ is called the $hyperspace$ $of$ $X$. A $Whitney$
$map$  is a map $\mu : C(X) \to [0,\mu(X)]$ satisfying  $\mu(\{x\})=0$ for each $x \in X$ and $\mu(A) < \mu(B)$  whenever $A, B \in C(X)$ and $A \subsetneq B$. It is well-known  that for each Whitney map  $\mu: C(X) \to [0,\mu(X)]$   
and  each $t \in [0, \mu(X)]$,  $\mu^{-1}(t)$ is a continuum (\cite[Theorem 19.9]{illanes}). Each $\mu^{-1}(t)$ is called a $Whitney$ $level$.

A continuum $X$ is said be $aposyndetic$ if for any two distinct points $x,y \in X$, there exists a subcontinuum $T \subseteq X$ such that $x \in {\rm Int}_X T \subseteq T \subseteq X \setminus \{y\}$. We say that a continuum $X$ is  $semiaposyndetic$ if for any two distinct points $x,y \in X$, there exists a subcontinuum $T \subseteq X$ such that either "$x \in {\rm Int}_X T \subseteq T \subseteq X \setminus \{y\}$", or "$y \in {\rm Int}_X T \subseteq T \subseteq X \setminus \{x\}$". 
A continuum $X$ is called a {\it Wilder continuum} if for any three  distinct points $x,y,z \in X$, there exists a subcontinuum $C \subseteq X$ such that $x \in C$ and  $C$ contains exactly one of $y$ and $z$. Wilder introduced the notion of a Wilder continuum in \cite{wilder2}.\footnote{Wilder  named the continua as $C$-continua. However, we use the term Wilder continua following the notion proposed in \cite{kk}.} Espinoza and the first author proved that each semiaposyndetic continuum is Wilder (\cite[Theorem 3.6]{D}).  

Let $X$ be a continuum. Then, $X$ is called a $D$-$continuum$ if for each disjoint nondegenerate subcontinua $A,B \subseteq X$, there exists a subcontinuum $C \subseteq X$ such that $A \cap C \neq \emptyset \neq B \cap C$ and $(A \cup B) \setminus C \neq \emptyset$. 
In addition, if we require $A\setminus C\neq\emptyset$ and $B\setminus C\neq\emptyset$, then the continuum $X$ is called a {\it $D^*$-continuum}.

Note that every arcwise connected continuum is Wilder and $D^*$. It is clear that every $D^*$-continuum is a $D$-continuum. Also, each Wilder continuum is a $D$-continuum \cite[Theorem 5.18]{loncar4}. A continuum
is said to be $decomposable$ if it is  the union of two proper subcontinua. If a continuum $X$ is not decomposable, then $X$ is said to be $indecomposable$.  Since every nondegenerate indecomposable continuum has uncountably many composants \cite[Theorem 11.15]{nadler1}, we can  easily  see that every nondegenerate $D$-continuum is decomposable. For other relationships between the classes of the above  continua and other classes of continua, see \cite[Figure  6]{D}. 

$D$-continua have some nice properties, see \cite{wwp}, \cite{D}, \cite{loncar1}, \cite{loncar2}, \cite{loncar3}, and \cite{loncar4}. Furthermore,   $D$-continua are related to   continua introduced by  Janiszewski \cite{jani}. It is known that Janiszewski’s continua are arc-like, arcless and hereditarily decomposable. 
In \cite[Section 6]{D}, Espinoza and the first author proved that there exists  Janiszewski's continuum which is a hereditarily $D$-continuum and contains neither a Wilder continuum nor a $D^*$-continuum. On the other hand, by \cite[Theorem 1]{wilder2} we can see that each arc-like Wilder continuum is an arc.  In connection with these results,  in \cite[Question 2]{D} Espinoza and the first author asked the existence of a continuum which is arc-like, $D^*$ and not an arc.

A topological property $P$ is called a {\it Whitney property} if a continuum $X$ has property $P$, so does $\mu^{-1}(t)$ for each Whitney map  $\mu$ for $C(X)$ and  each $t \in [0, \mu(X))$.  
 With respect to this property, many researches have been done so far. For example, arcwise connectedness, aposyndesis, semiaposyndesis, decomposability are Whitney properties (see  \cite[Theorem 3.4]{kra}, \cite[Theorem 14.8]{nadler2}, \cite[Propositions 9 and 12]{petrus}).  
 For further information about Whitney properties, see \cite[Chapter 8]{illanes}.

In this paper, 
we introduce the new class of continua; {\it $D^{**}$-continua}. In section 2, we give the definition of a $D^{**}$-continuum. Also, we show that the classes of   Wilder continua and   $D^{*}$-continua are strictly contained in the class of  $D^{**}$-continua, and  the class of  $D$-continua is bigger than the class of $D^{**}$-continua.    In section 3, using $D^{**}$-continua, we give the negative answer to \cite[Question 2]{D}. In  section 4, we prove that being Wilder, being $D$, being $D^*$ and being $D^{**}$ are Whitney properties. In section 5, we pose some questions.

\section{$D^{**}$-continua}
In this section, we introduce the notion of a $D^{**}$-continuum. Also, we  investigate the relationships between the class of $D^{**}$-continua and other classes of continua. 


\begin{definition}
Let $X$ be a continuum. Then, $X$ is called a $D^{**}$-$continuum$ if  for each disjoint nondegenerate subcontinua $A,B \subseteq X$, there exists a subcontinuum $C \subseteq X$ such that $A \cap C \neq \emptyset \neq B \cap C$ and $B \setminus C \neq \emptyset$.
\end{definition}

It is easy to see  that a continuum $X$ is a $D^{**}$-continuum if and only if for each $a \in X$ and each subcontinuum $B \subseteq X$ with $a \notin B$, there exists a subcontinuum $C \subseteq X$ such that $a \in C$, $B \cap C \neq \emptyset$ and $B \setminus C \neq \emptyset$.



\smallskip

The following proposition is easy to prove. Hence, we omit the proof.

\begin{proposition}
Every Wilder continuum and every $D^*$-continuum are $D^{**}$. Also, every $D^{**}$-continuum is $D$.

\label{D**}
\end{proposition}

Note that the $\sin (\frac{1}{x})$-continuum is a $D$-continuum which is not a $D^{**}$-continuum. Hence, the class of $D$-continua is bigger than the class of $D^{**}$-continua. 

By the following example, we see that there exists a $D^{**}$-continuum which is neither Wilder nor $D^*$. Therefore, the classes of Wilder continua and $D^{*}$-continua are strictly contained in the class of $D^{**}$-continua.

\begin{example}
Let $X=\{(t,{\rm sin}(\frac{1}{t})+2)\in \mathbb{R}^2~|~0<t\leq  6\}\cup (\{0\}\times [-{\rm sin}(1)-2,3])\cup ([0,-{\rm cos}(1)+2]\times \{-{\rm sin}(1)-2\})\cup \{(-t-|{\rm cos}(\frac{1}{t})|+3,-{\rm sin}(\frac{1}{t})-2)\in \mathbb{R}^2~|~0<t\leq 1\}\cup \{({\rm cos}(t)+3,{\rm sin}(t)-2)\in \mathbb{R}^2~|~0\leq t\leq 2\pi\}\cup \{(t+|{\rm cos}(\frac{1}{t})|+3,-{\rm sin}(\frac{1}{t})-2)\in \mathbb{R}^2~|~0<t\leq 1\}\cup ([{\rm cos}(1)+4,5]\times \{-{\rm sin}(1)-2\})\cup \{(-t+6,-{\rm sin}(\frac{1}{t})-2)\in \mathbb{R}^2~|~0<t\leq 1\}\cup (\{6\}\times [-3,{\rm sin}(\frac{1}{6})+2])$ (see Figure \ref{fig1}).  There are four arc components in $X$.  It is not difficult to see that $X$ is $D^{**}$.

To see $X$ is not Wilder, take $x,y,z \in X$ as in Figure \ref{fig1}. Then, there does not exist  a subcontinuum of $X$ containing $x$ and exactly one of $y$ and $z$. Hence, $X$ is not Wilder.

To see $X$ is not $D^{*}$, take  subcontinua $A, B  \subseteq X$ as in Figure \ref{fig1}. Then, for each subcontinuum $C \subseteq X$ with $A \cap C \neq \emptyset \neq B \cap C$, $A \setminus C= \emptyset$ or $B \setminus C= \emptyset$. Hence, $X$ is not $D^{*}$. 
\label{exampleD**}
\end{example}






\begin{figure}

\includegraphics[scale=0.40]{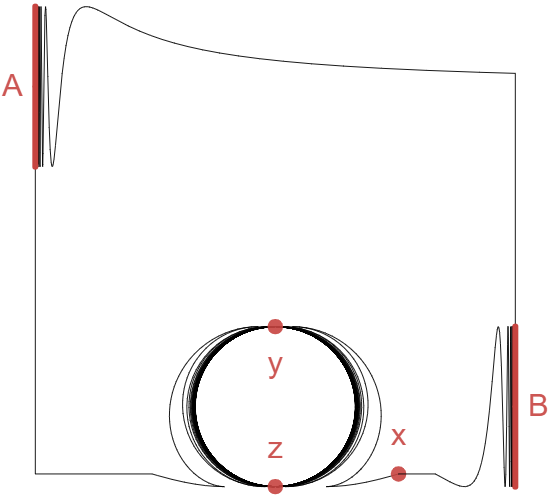}
\caption{A $D^{**}$-continuum which is neither Wilder nor $D^*$. }
\label{fig1}
\end{figure}


\section{Irreducible continua}


In this section, we give the negative answer to \cite[Question 2]{D}. First, we prove Theorem \ref{main} which is related to \cite[Theorem 1]{wilder2}.  Target spaces dealt in \cite[Theorem 1]{wilder2} is not necessarily metrizable.  However, Theorem \ref{main} is stronger than  \cite[Theorem 1]{wilder2} when  target spaces are restricted to  metrizable spaces.

Let $X$ be a nondegenerate continuum. Then $X$ is said to be {\it arc-like} if for each $\varepsilon > 0$, there exists a surjective map $f : X \to [0,1]$ such that for each $y \in [0,1]$, ${\rm diam} f^{-1}(y) < \varepsilon$. Also, $X$ is called a  \emph{irreducible continuum}  if there exist two points $a,b \in X$ such that for each subcontinuum $C \subseteq X$ with $a, b \in C$, $X=C$. In this 
case, we say that {\it $X$ is irreducible between $a$ and $b$}.  Note that each arc-like continuum is irreducible (\cite[Theorem 12.5]{nadler1}). 

Let $X$ be a nondegenerate continuum. A nondegenerate subcontinuum $C \subseteq X$ is called a {\it convergence continuum of $X$} provided that there exists a sequence $\{C_i\}_{i=1}^\infty$ of subcontinua of $X$ such that $\lim\limits_{i \to \infty}C_i=C$ and $C \cap C_i= \emptyset$ for each $i \ge 1$ (see \cite[Definition 5.11]{nadler1}).
\begin{theorem}
If  $X$ is an  irreducible $D^{**}$-continuum, then $X$ is an arc.
\label{main}
\end{theorem}

\begin{proof} Let $a$ and $b$ be points of irreducibility of $X$. 
Assume there exists a convergence continuum $C \subseteq X$. If $a, b \in C$, then $X=C$. This is a contradiction. Hence, we may assume $a \notin C$. Then, there exists a sequence  $\{C_i\}_{i=1}^{\infty}$ of subcontinua of $X$ such that $\lim\limits_{i\to\infty} C_i =C$ and $C \cap C_i = \emptyset$ for each $i \ge 1$. Since $X$ is $D^{**}$, there exists a subcontinuum $K \subseteq X$ such that $a \in K,$ $K \cap C \neq  \emptyset$ and $C \setminus K \neq \emptyset$. Take a subcontinuum $C' \subseteq C \setminus (K \cup \{b\})$. Since $X$ is $D^{**}$, there exists a subcontinuum $L \subseteq X$ such that $b \in L, ~  L \cap C' \neq \emptyset$ and $C' \setminus L \neq \emptyset$. Note that $K \cup C \cup L \subseteq X$ is a subcontinuum containing $a$ and $b$. Hence, $X = K \cup C \cup L$. Since $C \cap C_i = \emptyset$ for each $i \ge 1$,  $\bigcup_{i=1}^\infty C_i \subseteq K \cup L$. Therefore, $C \subseteq K \cup L$. This is a contradiction since $(C' \setminus L) \cap (K \cup L) = \emptyset$ and $C' \setminus L \subseteq C$.  Hence, $X$ does not contain a convergence continuum. By \cite[Theorem 10.4]{nadler1}, we see that $X$ is a locally connected continuum. Since $X$ is irreducible, by \cite[\S 50, II, Theorem 6]{kuratowski}, it follows that $X$ is an arc.  \end{proof}



Here, we introduce \cite[Question 2]{D}. 

\begin{question}{\rm (\cite[Question 2]{D})}
Is there an arc-like $D^*$-continuum which is not an arc ?

\end{question}

By Proposition \ref{D**} and Theorem \ref{main}, we can easily see the following corollary. Since each arc-like continuum is irreducible, the following result  gives the negative answer to \cite[Question 2]{D}.

\begin{corollary}{\rm (cf. \cite[Theorem 1]{wilder2})}
Let $X$ be a nondegenerate continuum. Then, the following are equivalent. 

\begin{itemize}
    \item[(1)] $X$ is an irreducible Wilder continuum,
    \item[(2)] $X$ is an irreducible $D^*$-continuum,
    \item[(3)] $X$ is an irreducible $D^{**}$-continuum, and 
    \item[(4)] $X$ is an arc.
    
\end{itemize}

\label{maincor}

\end{corollary}
Let $X$ be a continuum. Then, $X$ is said to be {\it hereditarily arcwise connected} if each of its subcontinuum is arcwise connected. Also, if each subcontinuum of $X$ is  Wilder (resp.   $D$,  $D^*$, $D^{**}$), then $X$ is called  {\it a hereditarily Wilder continuum} (resp. {\it  a hereditarily $D$-continuum, a hereditarily $D^*$-continuum, a hereditarily $D^{**}$-continuum}).

Note that each subcontinuum of an arc-like continuum is also arc-like. Hence, by using  Corollary \ref{maincor} and \cite[Corollary 5.12]{D}, we can get the following result which is  
 in contrast to the previous corollary.  

\begin{corollary}
There exists a Janiszewski continuum which is a  heredirarily $D$-continuum and contains no $D^{**}$-continua.  

\end{corollary}



 For each continuum $X$ and  each two distinct points $a,b \in X$, by using Zorn's Lemma, we see that there exists a subcontinuum of $X$ which is irreducible between $a$ and $b$. Hence, we can see the following result.
\begin{corollary}{\rm (cf. \cite[Corollary 2]{wilder2})}
Let $X$ be a nondegenerate continuum. Then, the following are equivalent. 

\begin{itemize}
    \item[(1)] $X$ is hereditarily Wilder,
    \item[(2)] $X$ is hereditarily $D^*$,
    \item[(3)] $X$ is hereditarily $D^{**}$, and 
    \item[(4)] $X$ is hereditarily arcwise connected.
    
\end{itemize}
\label{corhere}
\end{corollary}



    
    




    


\section{Whitney properties}

In this section, we prove that being Wilder, being $D$, being $D^{*}$ and being $D^{**}$ are Whitney properties. 

First, we introduce the following lemma which is used throughout this section.

\begin{lemma}{\rm (\cite[Lemma 14.8.1]{nadler2})}
Let $X$ be a continuum, let $\mu:C(X) \to [0,\mu(X)]$ be a Whitney map and let $t \in (0, \mu(X))$. Then, for each distinct elements $A,B \in \mu^{-1}(t)$ with $A \cap B \neq \emptyset$,  there exists an arc $\mathcal{I}_{A,B} \subseteq \mu^{-1}(t)$ from $A$ to $B$ such that for each $L \in \mathcal{I}_{A,B}, ~ L \subseteq A \cup B$. Moreover, if $K$ is any given component of $A \cap B$, then   $\mathcal{I}_{A,B}$ may be chosen as above so that $K \subseteq L$ for each $L  \in \mathcal{I}_{A,B}$.
\label{lemmanadler}
\end{lemma}




\begin{theorem}
Being Wilder is a Whitney property.

\end{theorem}

\begin{proof} 
 Let $X$ be a Wilder continuum, let  $\mu: C(X) \to [0,\mu(X)]$ be a Whitney map and  let $t \in (0,\mu(X))$. Let $A,B,C \in \mu^{-1}(t)$ be such that $A \neq B \neq C \neq A$. We will find a subcontinuum of $\mu^{-1}(t)$ containing $A$ and exactly one of $B$ and $C$.
 
 \smallskip
 
\begin{itemize}

 \item[Case 1.] $A \subseteq B \cup C$.
    
    Take any $a \in A \setminus B$. Then, $a \in (A \cap C) \setminus B$. By Lemma \ref{lemmanadler}, there exists an arc $\mathcal{I}_{A,C} \subseteq \mu^{-1}(t)$ from $A$ to $C$ such that for each $L \in \mathcal{I}_{A,C}$, $a \in L$. Then, $A, C \in \mathcal{I}_{A,C}$ and $B \notin  \mathcal{I}_{A,C}$.

    
    \smallskip

 \item[Case 2.]  $B \subseteq A \cup C$ or $C \subseteq A \cup B$. 
 
We may assume that $B \subseteq A \cup C$.
Take any $b \in B \setminus C$. Then, $b \in (A \cap B) \setminus C$. By Lemma \ref{lemmanadler}, there exists an arc $\mathcal{I}_{A,B} \subseteq \mu^{-1}(t)$ from $A$ to $B$ such that for each $L \in \mathcal{I}_{A,B}$, $b \in L$. Then, $A, B \in \mathcal{I}_{A,B}$ and $C \notin  \mathcal{I}_{A,B}$. 

  \smallskip

    \item [Case 3.] $A \nsubseteq B \cup C, ~ B \nsubseteq A \cup C$ and  $C \nsubseteq A \cup B$.
    
    \smallskip
    
    Let $a \in A \setminus (B\cup C)$, $b \in B \setminus (A\cup C)$ and $c \in C \setminus (A\cup B)$. Since $X$ is Wilder, we may assume there exists a subcontinuum $K \subseteq X$ such that $a,b \in K$ and $c \notin K$. Let 
    $\mathcal{L}=\{L \in \mu^{-1}(t) \ | \ L \subseteq A \cup B \cup K\}$.  
    Then,  $\mathcal{L} \subseteq \mu^{-1}(t)$ is a subcontinuum 
 satisfying $A,B \in \mathcal{L}$ and $C \notin \mathcal{L}$ (see \cite[Exercise 27.7]{illanes}).

    \smallskip

\end{itemize}

 In every case above, we can find a subcontinuum of $\mu^{-1}(t)$ containing $A$ and exactly one of $B$ and $C$.
  Hence, $\mu^{-1}(t)$ is Wilder. This completes the proof.
\end{proof}

\begin{theorem}
Being  $D$ is a Whitney property.
\label{wd}
\end{theorem}

\begin{proof}
Let $X$ be a $D$-continuum, let  $\mu: C(X) \to [0,\mu(X)]$ be a Whitney map and  let $t \in (0,\mu(X))$. Let  $\mathcal{A}$, 
$\mathcal{B} \subseteq \mu^{-1}(t)$ be  pairwise disjoint nondegenerate subcontinua.

 \begin{itemize}
 
 \item[Case 1.]  $(\bigcup \mathcal{A}) \cap (\bigcup \mathcal{B}) \neq \emptyset$.
 
    In this case, there exist $A_0 \in \mathcal{A}$ and $B_0 \in \mathcal {B}$ such that  $A_0 \cap B_0 \neq \emptyset$. Then, by Lemma \ref{lemmanadler}, 
    there exists an arc $\mathcal{I}_{A_0, B_0} \subseteq \mu^{-1}(t)$ from $A_0$ to $B_0$. 
      Since $\mathcal{A} \cap \mathcal{I}_{A_0, B_0}$ and $\mathcal{B} \cap \mathcal{I}_{A_0, B_0}$ are nonempty,  pairwise disjoint and closed in $\mathcal{I}_{A_0, B_0}$, there exist $A' \in \mathcal{A} \cap  {I}_{A_0, B_0}$,  $B' \in \mathcal{B} \cap {I}_{A_0, B_0}$ and the subarc $\mathcal{I}_{A', B'} \subseteq \mathcal{I}_{A_0, B_0}$ from $A'$ to $B'$ such that $(\mathcal{I}_{A', B'} \setminus \{A', B'\}) \cap (\mathcal{A} \cup \mathcal{B}) = \emptyset$.  Then, $\mathcal{I}_{A', B'} \subseteq \mu^{-1}(t)$ is a subcontinuum satisfying $\mathcal{A} \cap \mathcal{I}_{A', B'}  \neq \emptyset \neq \mathcal{B} \cap \mathcal{I}_{A', B'}$ and $\mathcal{A} \setminus \mathcal{I}_{A', B'}  \neq \emptyset \neq \mathcal{B} \setminus  \mathcal{I}_{A', B'}$.

     \medskip
      \item[Case 2.] $(\bigcup \mathcal{A}) \cap (\bigcup \mathcal{B}) = \emptyset$.
      
      \smallskip
      
      Take any $A \in \mathcal{A}$ and $B \in \mathcal{B}$. Since  $A \cap B = \emptyset$ and $X$ is a $D$-continuum, there exists a subcontinuum $C \subseteq X$ such that $A \cap C \neq \emptyset \neq B \cap C$ and $(A \cup B) \setminus C \neq \emptyset$. We may assume that $B \setminus C \neq \emptyset$. Then, it is easy to see that $\mathcal{L}_A=\{ L \in \mu^{-1}(t) \ | \ L \subseteq A \cup C\}$ and  $\mathcal{L}=\{ L \in \mu^{-1}(t) \ | \ L \subseteq A \cup B \cup C\}$ are subcontinua of $C(X)$ (see \cite[Exercise 27.7]{illanes}). Note that $A \in \mathcal{L}_A, ~ B \notin \mathcal{L}_A$ and $B \in \mathcal{L}$.

     \begin{itemize}
     
     \smallskip
     
         \item[Case 2.1.] $\mathcal{B} \cap \mathcal{L}_A \neq \emptyset$.
         
         \smallskip
         
          In this case,  
          we can easily see that $\mathcal{L}_A \subseteq \mu^{-1}(t)$ is a subcontinuum satisfying 
         $\mathcal{A} \cap \mathcal{L}_A  \neq \emptyset  \neq \mathcal{B} \cap \mathcal{L}_A  \neq \emptyset$ and $\mathcal{B} \setminus \mathcal{L}_A \neq \emptyset$.
     
          \smallskip
     
     \item[Case 2.2.] $\mathcal{B} \cap \mathcal{L}_A = \emptyset$. 
     
     \smallskip
     
     \begin{itemize}
         \item[Case 2.2.1] $\mathcal{B} \cap \mathcal{L} = \{B\}$. 
         
         \smallskip
         
         In this case, it is easy to see that  $\mathcal{L} \subseteq \mu^{-1}(t)$ is a subcontinuum satisfying   $\mathcal{A} \cap \mathcal{L} \neq \emptyset \neq \mathcal{B} \cap \mathcal{L}$ and $\mathcal{B} \setminus \mathcal{L} \neq \emptyset$.
         
         \smallskip
         
           \item[Case 2.2.2] $\mathcal{B} \cap \mathcal{L} \neq \{B\}$.

           \smallskip
           
            In this case, there exists $B_1 \in   (\mathcal{B} \cap \mathcal{L} )\setminus \{B\} $. Since $\mathcal{B} \cap \mathcal{L}_A = \emptyset$, $B_1 \cap (B \setminus C) \neq \emptyset$. Since $B,B_1 \in \mu^{-1}(t)$ and $B \neq B_1$, there exists a point $b \in ((A\cup C) \cap B_1) \setminus B$. By \cite[Theorem 14.6]{illanes}, we can take a subcontinuum $B_2 \subseteq X$ such that $B_2 \in \mu^{-1}(t)$ and  $b \in B_2 \subseteq A \cup C$. Since $\mathcal{B} \cap \mathcal{L}_A = \emptyset$, $B_1 \in \mathcal{B}$ and $B_2 \in \mathcal{L}_A,$ we can see that $B_1 \neq B_2$. Then, by Lemma \ref{lemmanadler},  there exists an arc $\mathcal{I}_{B_1, B_2} \subseteq \mu^{-1}(t) $ from $B_1$ to $B_2$ such that for each $L \in \mathcal{I}_{B_1, B_2}, ~ b \in L$. Note that $B \in \mathcal{B} \setminus (\mathcal{L}_A \cup \mathcal{I}_{B_1, B_2}) $. Hence, we see that $\mathcal{L}_A \cup \mathcal{I}_{B_1, B_2} \subseteq \mu^{-1}(t)$ is a subcontinuum satisfying  $\mathcal{A} \cap (\mathcal{L}_A \cup \mathcal{I}_{B_1, B_2})   \neq \emptyset \neq \mathcal{B} \cap (\mathcal{L}_A \cup \mathcal{I}_{B_1, B_2}) $ and $\mathcal{B} \setminus (\mathcal{L}_A \cup \mathcal{I}_{B_1, B_2}) \neq \emptyset$.

     \end{itemize}

     \end{itemize}
     
 \end{itemize}

\smallskip

     Hence, we see that $\mu^{-1}(t)$ is a $D^{}$-continuum. This completes the proof.
\end{proof}

The proof of the following result is similar to the proof of Theorem \ref{wd}. Hence, we omit the proof. 

\begin{theorem}
Being  $D^{**}$ is a Whitney property.
\label{wd*}
\end{theorem}












Finally, we prove the following result.

\begin{theorem}
Being  $D^{*}$ is a Whitney property.

\end{theorem}

\begin{proof}
Let $X$ be a $D^{*}$-continuum, let  $\mu: C(X) \to [0,\mu(X)]$ be a Whitney map and  let $t \in (0,\mu(X))$. Let  $\mathcal{A}$, 
$\mathcal{B} \subseteq \mu^{-1}(t)$ be  pairwise disjoint nondegenerate subcontinua.

 \begin{itemize}
     \item[Case 1.] $(\bigcup \mathcal{A}) \cap (\bigcup \mathcal{B}) \neq \emptyset$.
     
    \smallskip 
    
    This case is similar to Case 1 in the proof of Theorem \ref{wd}. 
    
     \smallskip

      \item[Case 2.] $(\bigcup \mathcal{A}) \cap (\bigcup \mathcal{B}) = \emptyset$.
      
      \smallskip
      
      Take any $A \in \mathcal{A}$ and $B \in \mathcal{B}$. Since $A \cap B = \emptyset$ and $X$ is $D^{*}$, there exists a subcontinuum $C \subseteq X$ such that $A \cap C \neq \emptyset \neq B \cap C$ and  $A \setminus C \neq \emptyset \neq B \setminus C$.   Let  $\mathcal{L}_A=\{ L \in \mu^{-1}(t) \ | \ L \subseteq A \cup C \}$, $\mathcal{L}_B=\{ L \in \mu^{-1}(t) \ | \ L \subseteq B \cup C \}$ and $\mathcal{L}=\{ L \in \mu^{-1}(t) \ | \ L \subseteq A \cup B \cup C\}$. Then, 
      $\mathcal{L}_A$, $\mathcal{L}_B$ and $\mathcal{L}$ are subcontinua of $C(X)$ (see \cite[Exercise 27.7]{illanes}). Note that $\mathcal{A} \cap \mathcal{L}_A = \mathcal{A} \cap \mathcal{L}$ and $\mathcal{B} \cap \mathcal{L}_B = \mathcal{B} \cap \mathcal{L}$.

     \medskip
      
      \begin{itemize}
      
      \item[Case 2.1.] $\mu(C) < t.$
      
      \smallskip
      
      \begin{itemize}
          \item[Case 2.1.1.] $\mathcal{A} \cap \mathcal{L}_A  = \{A\}$ and  
          $\mathcal{B} \cap \mathcal{L}_B  = \{B\}$.

          \smallskip
          
          In this case, $\mathcal{L} \subseteq \mu^{-1}(t)$ is a subcontinuum such that $\mathcal{A} \cap \mathcal{L} \neq \emptyset \neq \mathcal{B} \cap \mathcal{L}$ and $\mathcal{A} \setminus \mathcal{L} \neq \emptyset \neq \mathcal{B} \setminus \mathcal{L}$.

          \smallskip

         \item[Case 2.1.2.]  $ \mathcal{A} \cap \mathcal{L}_A \neq  \{A\}$ and  
         $ \mathcal{B} \cap \mathcal{L}_B \neq  \{B\}$.
         
         \smallskip
         
         Take $A_1 \in (\mathcal{A} \cap \mathcal{L}_A) \setminus \{A\}$  and $B_1 \in ( \mathcal{B} \cap \mathcal{L}_B) \setminus \{B\}$.  Since $\{A,A_1,B,B_1\} \subseteq \mu^{-1}(t)$, $A \neq A_1$ and $B \neq B_1$, there exist  points $a \in (A_1 \cap C) \setminus A$ and  $b \in (B_1 \cap C) \setminus B$.   Since $0< \mu(C) < t$, by the proof of Lemma \ref{lemmanadler} (i.e. the proof of \cite[Lemma 14.8.1]{nadler2}), there exist subcontinua $A' \subsetneq A$ and $B' \subsetneq B$ such that $\mu(A' \cup C) = t = \mu(B' \cup C)$. Then, by Lemma \ref{lemmanadler}, there exist an   arc $\mathcal{I}_{A' \cup C, A_1} \subseteq \mu^{-1}(t)$ from $A' \cup C$ to $A_1$ and an arc $  \mathcal{I}_{B' \cup C, B_1} \subseteq \mu^{-1}(t)$ from $B' \cup C$ to $B_1$ such that   for each $L \in \mathcal{I}_{A' \cup C, A_1}$ and $K \in \mathcal{I}_{B' \cup C, B_1}$, $a \in L$ and $b \in K$.  
         By  Lemma \ref{lemmanadler}, there exists an arc $\mathcal{I}_{A' \cup C, B' \cup C} \subseteq \mu^{-1}(t)$ from $A' \cup C$ to $ B' \cup C$ such that for each $L \in \mathcal{I}_{A' \cup C, B' \cup C}$, $C \subseteq L$. Let  $\mathcal{C}=\mathcal{I}_{A' \cup C, A_1} \cup \mathcal{I}_{A' \cup C, B' \cup C} \cup \mathcal{I}_{B' \cup C, B_1} \subseteq \mu^{-1}(t)$. Note that $A, B \notin \mathcal{C}$. Then, $\mathcal{C} \subseteq \mu^{-1}(t)$ is a subcontinuum such that $\mathcal{A} \cap \mathcal{C} \neq \emptyset \neq \mathcal{B} \cap \mathcal{C}$ and  
          $\mathcal{A} \setminus \mathcal{C} \neq \emptyset \neq \mathcal{B} \setminus \mathcal{C}$.
          
          \smallskip
          
          \item[Case 2.1.3.]  $\mathcal{A} \cap \mathcal{L}_{A} \neq  \{A\}$ and  $\mathcal{B} \cap \mathcal{L}_{B} = \{B\}$.
          
          \smallskip 
          
           Take $A_1 \in ( \mathcal{A} \cap \mathcal{L}_A) \setminus \{A\}$, $a \in (A_1 \cap C) \setminus A$, $A' \subseteq A$ and $\mathcal{I}_{A' \cup C, A_1} \subseteq \mu^{-1}(t)$ as in the previous case. 
            By  Lemma \ref{lemmanadler},  there exists an arc $\mathcal{I}_{A' \cup C, B} \subseteq \mu^{-1}(t)$ from $A' \cup C$ to $ B$ such that for each $L \in \mathcal{I}_{A' \cup C, B }$, $L \subseteq A' \cup C \cup B$. Let  $\mathcal{C}=\mathcal{I}_{A' \cup C, A_1} \cup \mathcal{I}_{A' \cup C, B}$. Note that $A \notin \mathcal{C}$ and $\mathcal{B} \cap \mathcal{C}=\{B\}$. Then, $\mathcal{C} \subseteq \mu^{-1}(t)$ is a subcontinuum such that $\mathcal{A} \cap \mathcal{C} \neq \emptyset \neq \mathcal{B} \cap \mathcal{C}$ 
           and  
          $\mathcal{A} \setminus \mathcal{C} \neq \emptyset \neq \mathcal{B} \setminus \mathcal{C}$.

          \smallskip
          
          \item[Case 2.1.4.]  $ \mathcal{A} \cap \mathcal{L}_A =  \{A\}$ and  $\mathcal{B} \cap \mathcal{L}_B \neq \{B\}$.

          \smallskip
            This case is similar to Case 2.1.3.

      \end{itemize}

      \smallskip

          \item[Case 2.2.]  $\mu(C) \ge t$.

     \smallskip
     
     Let $\mathcal{L}_C= \{L \in \mu^{-1}(t) \ | \ L \subseteq C \}$. Note that 
     $A,B \notin \mathcal{L}_C$.

     \begin{itemize}
         
     \smallskip
     
       \item[Case 2.2.1.] $\mathcal{A} \cap \mathcal{L}_A = \{A\}$ and  $\mathcal{B} \cap \mathcal{L}_{B} = \{B\}$.

          \smallskip
          
          In this case, $\mathcal{L} \subseteq \mu^{-1}(t)$ is a subcontinuum such that $\mathcal{A} \cap \mathcal{L} \neq \emptyset \neq \mathcal{B} \cap \mathcal{L}$ and $\mathcal{A} \setminus \mathcal{L} \neq \emptyset \neq \mathcal{B} \setminus \mathcal{L}$.

          \smallskip

         \item[Case 2.2.2.]  $\mathcal{A} \cap \mathcal{L}_{A} \neq  \{A\}$ and  $\mathcal{B} \cap \mathcal{L}_{B} \neq  \{B\}$.
         
         \smallskip
         
         Take $A_1 \in (\mathcal{A} \cap \mathcal{L}_{A}) \setminus \{A\}$  and $B_1 \in (\mathcal{B} \cap \mathcal{L}_{B}) \setminus \{B\}$. 
         
         \smallskip
         
         \begin{itemize}

         \item[Case 2.2.2.1.]$A_1 \cup  B_1 \subseteq C$. 
         
         \smallskip
         
        In this case, 
     $\mathcal{L}_C \subseteq \mu^{-1}(t)$ is a subcontinuum such that  $\mathcal{A} \cap \mathcal{L}_C \neq \emptyset \neq \mathcal{B} \cap \mathcal{L}_C$ and $\mathcal{A} \setminus \mathcal{L}_C \neq \emptyset \neq \mathcal{B} \setminus \mathcal{L}_C$.

         \smallskip
         
         \item[Case 2.2.2.2.] $A_1 \nsubseteq C$ and $B_1 \nsubseteq C$.

         \smallskip
         
          Since $\{A,A_1,B,B_1\} \subseteq \mu^{-1}(t)$, $A \neq A_1$ and $B \neq B_1$, there exist  points $a \in (A_1 \cap C) \setminus A$ and  $b \in (B_1 \cap C) \setminus B$.   Since $\mu(C) \ge t$, there exist subcontinua $A' \subseteq C$ and $B' \subseteq C$ such that $a \in A'$, $b \in B'$ and $\mu(A' ) = t = \mu(B' )$ (see \cite[Theorem 14.6]{illanes}). Then, by Lemma \ref{lemmanadler} there exist  an arc $\mathcal{I}_{A_1, A'} \subseteq \mu^{-1}(t)$ from $A_1$ to $A'$ and an arc $\mathcal{I}_{B_1, B'} \subseteq \mu^{-1}(t)$ from $B_1$ to $B'$ such that for each $L \in \mathcal{I}_{A_1, A'}$ and $K \in \mathcal{I}_{B_1, B'}$, $a \in L$ and $b \in K$.  Let $\mathcal{C}=\mathcal{I}_{A_1,A'} \cup \mathcal{L}_C \cup \mathcal{I}_{B_1,B'}$. Note that $A, B \notin \mathcal{C}$. Then,  $ \mathcal{C} \subseteq \mu^{-1}(t)$ is a subcontinuum such that $\mathcal{A} \cap \mathcal{C} \neq \emptyset \neq \mathcal{B} \cap \mathcal{C}$ and  
          $\mathcal{A} \setminus \mathcal{C} \neq \emptyset \neq \mathcal{B} \setminus \mathcal{C}$.
          
          \smallskip
          
          \item[Case 2.2.2.3.]  $A_1 \nsubseteq C$ and $B_1 \subseteq C$. 
          
          \smallskip 
          
           Take  $a \in (A_1 \cap C) \setminus A$, $A' \subseteq C$ and $\mathcal{I}_{A_1, A'} \subseteq \mu^{-1}(t)$ as in the previous case.  Let $\mathcal{C}=\mathcal{I}_{A_1,A'} \cup \mathcal{L}_C$. Then, $\mathcal{C} \subseteq \mu^{-1}(t)$ is a subcontinuum such that $\mathcal{A} \cap \mathcal{C} \neq \emptyset \neq \mathcal{B} \cap \mathcal{C}$  
           and  
          $\mathcal{A} \setminus \mathcal{C} \neq \emptyset \neq \mathcal{B} \setminus \mathcal{C}$.

          \smallskip
          
          \item[Case 2.2.2.4.]  $A_1 \subseteq C$ and $B_1 \nsubseteq C$.
          
          \smallskip
This case is similar to Case 2.2.2.3.

     \end{itemize}

\smallskip

     \end{itemize}
    
     \begin{itemize}
    \item[Case 2.2.3.] $\mathcal{A} \cap \mathcal{L}_{A} \neq  \{A\}$ and  $\mathcal{B} \cap \mathcal{L}_{B} =  \{B\}$.

             \smallskip 
             
              Take $A_1 \in (\mathcal{A} \cap \mathcal{L}_{A}) \setminus \{A\}$.

         \smallskip
         
              \begin{itemize}
         
         \item[Case 2.2.3.1.]$A_1 \subseteq C$.

         \smallskip
         
        In this case, $\mathcal{L}_{B} \subseteq \mu^{-1}(t)$ is a subcontinuum satisfying  $\mathcal{A} \cap \mathcal{L}_B \neq \emptyset \neq \mathcal{B} \cap \mathcal{L}_B$ and $\mathcal{A} \setminus \mathcal{L}_B \neq \emptyset \neq \mathcal{B} \setminus \mathcal{L_B}$.

         \smallskip
         
         \item[Case 2.2.3.2.] $A_1 \nsubseteq C$.

         \smallskip
         Take  $a \in (A_1 \cap C) \setminus A$, $A' \subseteq C$ and $\mathcal{I}_{A_1, A'} \subseteq \mu^{-1}(t)$ as in Case 2.2.2.2.
           Let  $\mathcal{C}=\mathcal{I}_{A_1,A'} \cup \mathcal{L}_B.$ Then, $\mathcal{C} \subseteq \mu^{-1}(t)$ is a subcontinuum satisfying $\mathcal{A} \cap \mathcal{C} \neq \emptyset \neq \mathcal{B} \cap \mathcal{C}$ and  
          $\mathcal{A} \setminus \mathcal{C} \neq \emptyset \neq \mathcal{B} \setminus \mathcal{C}$.
          
          \smallskip

    \end{itemize}

\end{itemize}

         \begin{itemize}
    \item[Case 2.2.4.] $\mathcal{A} \cap \mathcal{L}_{A} =  \{A\}$ and  $\mathcal{B} \cap \mathcal{L}_{B} \neq  \{B\}$.

    \smallskip
    
    This case is similar to the case 2.2.3.
    
\end{itemize}   
         
     \end{itemize}

 \end{itemize}

\smallskip
In every case above, we can take a subcontinuum $\mathcal{C} \subseteq \mu^{-1}(t)$ such that $\mathcal{A} \cap \mathcal{C} \neq \emptyset \neq \mathcal{B} \cap \mathcal{C}$ and  
          $\mathcal{A} \setminus \mathcal{C} \neq \emptyset \neq \mathcal{B} \setminus \mathcal{C}$. 
     Hence, we see that $\mu^{-1}(t)$ is $D^{*}$. This completes the proof.
\end{proof}

\section{questions}

In this section, we pose some questions related to the previous section.   

\begin{definition} (\cite[Definition 27.1]{illanes})
Let $P$ be a topological property. Then, $P$ is called: 

\begin{itemize}
    \item  a {\it Whitney reversible property} provided that whenever $X$ is a continuum such that $\mu^{-1}(t)$ has property $P$ for each Whitney map $\mu$ for $C(X)$ and  each $t \in (0, \mu(X))$, then $X$ has property $P$.
    
    \item a {\it strong Whitney reversible property} provided that whenever $X$ is a continuum such that $\mu^{-1}(t)$ has property $P$ for some Whitney map $\mu$ for $C(X)$ and  each $t \in (0, \mu(X))$, then $X$ has property $P$. 
    
    \item a {\it sequential strong Whitney reversible property} provided that whenever $X$ is a continuum such that there exist a Whitney map $\mu$ for $C(X)$ and a sequence $\{t_n\}_{n=1}^\infty \subseteq (0,\mu(X))$ such that $\lim\limits_{n\to\infty} t_n = 0$ and $ \mu^{-1}(t_n)$ has property $P$ for each $n \ge 1$,  then $X$ has property $P$.
    
\end{itemize}
\end{definition}

As well as Whitney  properties, many researchers have studied these properties so far.  

In connection with the previous section, it is natural to ask the following.

\begin{questions}
Is being Wilder (being $D$, being $D^*$, being $D^{**}$) a Whitney reversible (strong Whitney reversible, sequential strong Whitney reversible) property ?

\end{questions}

\section{Acknowledgements} 
 This work was supported by JSPS KAKENHI Grant number 21K03249.

\end{document}